\documentclass[12pt, reqno]{article}
\usepackage[utf8]{inputenc}
\usepackage{amsthm, amsmath, amssymb, tikz, bm,diffcoeff}
\usepackage{verbatim}
\usepackage{graphicx}
\usepackage[linesnumbered,ruled,vlined]{algorithm2e}
\usepackage{float}
\usepackage{derivative}
\usepackage{mathtools}
\usepackage{CJK} 
\usepackage{xifthen}
\usepackage{bbm}
\usepackage[colorlinks=true,
linkcolor=blue,citecolor=blue,
urlcolor=blue]{hyperref}
\usetikzlibrary{arrows,decorations.markings}
\usepackage[margin=1.0in]{geometry}
\usepackage{extarrows}

\usepackage[shortlabels]{enumitem}
\usepackage{todonotes}
\usetikzlibrary{calc,shapes, backgrounds}
\allowdisplaybreaks

\newtheorem{lemma}{Lemma}
\newtheorem{theorem}{Theorem}

\newtheorem{proposition}{Proposition}
\newtheorem{example}{Example}
\usepackage{tikz}

\newtheorem{cor}
{Corollary}

\newtheorem{defn}
{Definition}

\usepackage{mathtools}

\newcommand{\R}{\mathbb{R}}
\newcommand{\II}{\mathcal{I}}

\DeclareMathOperator*{\argmin}{arg\,min}
\usetikzlibrary{shapes,backgrounds}
\usepackage{verbatim}
\makeatother
\usepackage[inline]{asymptote}
\usepackage{extarrows}
\usepackage{hyperref}
\hypersetup{
    colorlinks,
    citecolor=black,
    filecolor=black,
    linkcolor=black,
    urlcolor=black
}

\title{Combinatorial Courant-Fischer-Weyl Minimax Principle on 
Cheeger $k$-constants of Weighted Forests}
\author{Zijun Meng\thanks{Department of Mathematics, The Chinese University of Hong Kong,
Hong Kong, China. Email address: 1155194127@link.cuhk.edu.hk} \and Dong Zhang\thanks{School of Mathematical Sciences, Peking University,
100871 Beijing, China. Email address: dongzhang@math.pku.edu.cn}}

\date{}
\begin{document}

\maketitle

\begin{abstract}
We present several max-min and minimax characterizations of Cheeger $k$-constants in weighted forests, thereby providing the first combinatorial analogue of the Courant-\\Fischer-Weyl minimax principle. As for applications, we prove that the forest 1-Laplacian variational eigenvalues are independent of the choice of typical indexes; we propose a refined higher order Cheeger inequality which establishes a relationship between Miclo's multiway Cheeger constants and the variational eigenvalues of graph $p$-Laplacians; and we show a purely combinatorial proof for the equality $h_k=\lambda_k(\Delta_1)$ on trees which connects the 1-Laplacian variational eigenvalues and the multiway Cheeger constants. 
\end{abstract}
\tableofcontents

\section{Introduction}
The Cheeger $k$-constant is omnipresent in mathematics and computer science. It appears in different areas under different names: it is called the \emph{higher isoperimetric numbers}  \cite{Daneshgar-Hajiabolhassan-Javadi} or the \emph{$k$-way isoperimetric constant} \cite{Mimura} in graph theory, and the \emph{$k$-th Cheeger constant} \cite{Bobkov-Parini} or the \emph{connectivity spectrum} \cite{Hassannezhad-Miclo} in geometry and analysis. 
Its importance comes from its deep connection with the eigenvalues of the Laplacian, which in turn govern many properties of graphs, networks, and even continuous spaces such as manifolds. For example, the classical Cheeger inequality relates the first nontrivial eigenvalue of the Laplacian to the Cheeger constant, and higher-order versions of this inequality involve the Cheeger $k$-constant and the $k$-th eigenvalue. 
These inequalities are at the heart of spectral clustering, graph partitioning, and many modern data analysis techniques \cite{Lee-Gharan-Trevisan,Hassannezhad-Miclo}.  
Moreover, many of these results have been extended to the case of nonlinear eigenproblems \cite{Bungert}, particularly those of the $p$-Laplacian type \cite{Tudisco,Keller}.  

Despite its widespread use, the Cheeger $k$-constant has always been defined through a ``minimax'' process: we look for a collection of $k$ disjoint subsets that minimizes the maximum expansion among them. This definition is natural but it has no obvious dual \emph{maximin} or alternative characterizations. In sharp contrast, the eigenvalues of a real symmetric matrix—for example, the Laplacian—admit a beautiful \emph{minimax principle} (the Courant–Fischer–\\Weyl theorem, see, for instance, Corollary III.1.2 of \cite{Bhatia}), which says that the $k$-th eigenvalue can be expressed both as a min-max and as a max-min over subspaces. This duality is extremely powerful in both theory and applications. A natural question then arises: \textbf{Can we find a purely combinatorial analogue of the Courant–Fischer–Weyl minimax principle for the Cheeger $k$-constant?}

In this paper we give an affirmative answer to this question for the important class of \emph{forests} (i.e. graphs without cycles); and, more generally, we also provide an approximate solution to this problem for graphs with a small number of cycles. 

Before starting with the exposition of our results, let us fix some notions and notations. Recall that a \emph{(finite, undirected) weighted graph} $G$ is a quadruple $(V,E,\mu,w)$ with a \emph{vertex set} $V=\{1,\cdots,n\}$, an \emph{edge set} \[E\subseteq\{\{u,v\}:u\ne v\text{ in }V\},\] a \emph{vertex weight} $\mu:V\to \mathbb{R}^+$, and an \emph{edge weight} $w:E\to \R^+$. For simplicity, we use $w_{uv}$ and $\mu_v$ to denote $w(\{u,v\})$ and $\mu(v)$, respectively. 

Fix a weighted graph $G=(V,E,\mu,w)$. For any subset $A\subseteq V$ of vertices, define the \emph{boundary} $\partial A$ of $A$ by \[\partial A:=\{\{u,v\}\in E\colon u\in A, v\not\in A\}.\] 
The \emph{Cheeger $k$-constant} of $G$ is defined by \[h_k(G):=\min_{\text{subpartitions }(A_1,\cdots,A_k)\text{ of }V}~\max_{1\le i\le k}\varphi(A_i),\] where \[\varphi(A_i):=\varphi_{\mu,w}(A_i):=\frac{w(\partial A)}{\mu(A)}:=\frac{\sum\limits_{\{u,v\}\in\partial A}w_{uv}}{\sum\limits_{v\in A}\mu_v}\] is the \textsl{expansion} (Lee-Oveis Gharan-Trevisan, \cite{Lee-Gharan-Trevisan}) of $A_i$, and by $(A_1,\cdots,A_k)$ being a \emph{subpartition} of $V$ we mean that $A_i$ are pairwise disjoint nonempty subsets of $V$. 

For $p>1$, we define the $p$-Laplacian $\Delta_p:\mathbb R^V\to\mathbb R^V$ of the weighted graph $G=(V,\,E,\,\mu,\,w)$ by \[(\Delta_px)(v)=\frac1{\mu_v}\sum_{u\in V: u\sim v}w_{uv}|x_v-x_u|^{p-2}(x_v-x_u),\] where $u\sim v$ means $\{u,\,v\}\in E$.

Let $\lambda_{k}(\Delta_p)$ denote the $k$-th min-max eigenvalue of the graph $p$-Laplacian $\Delta_p$, where the relevant concepts and terminology are detailed in Section \ref{sec:minmax-eigen-1-Lap}. In the case $p=2$, $\Delta_2$ is nothing but the commonly used \emph{(normalized) graph Laplacian}, and $\lambda_{k}(\Delta_2)$ is the $k$-th smallest eigenvalue of $\Delta_2$.  
The most important inequality connecting $h_k(G)$ and $\lambda_k(\Delta_2)$ is the higher order Cheeger inequality proved by Lee, Oveis Gharan, and Trevisan \cite{Lee-Gharan-Trevisan}, which states that  there exists a universal constant $C>0$ such that
\begin{equation}\label{eq:LOT}
\frac{1}{Ck^4}h_k(G)^2\le\lambda_k(\Delta_2)\le2h_k(G).    
\end{equation}
Moreover, if $G$ is further assumed to be a tree, Miclo \cite{Miclo} as well as Daneshgar, Javadi and Miclo \cite{Daneshgar} show that $$\frac{1}{2}h_k(G)^2\le \lambda_{k}(\Delta_2)\le 2 h_k(G).$$
When $\Delta_2$ is replaced by $\Delta_p$, Tudisco and Hein \cite{Tudisco} establish the following higher order Cheeger inequality \begin{equation}
\label{eq:TH}
 \frac{2^{p-1}}{p^p}h_m(G)^p\le\lambda_k(\Delta_p)\le 2^{p-1}h_k(G)
\end{equation}
where $m$ is the number of strong nodal domains of an eigenvector corresponding to the eigenvalue $\lambda_k(\Delta_p)$, and $p\ge 1$. Again, if $G$ is further assumed to be a weighted tree, and $p=1$, Deidda, Putti and Tudisco \cite{Deidda} claim that $$\lambda_{k}(\Delta_1)= h_k(G).$$ 


Note that the factor $\frac{1}{Ck^4}$ in the lower bound of Lee-Oveis Gharan-Trevisan's higher order Cheeger inequality \eqref{eq:LOT} \textbf{depends on $k$}, and the subscript $m$ of $h_m(G)$ appearing in Tudisco-Hein's higher order Cheeger inequality \eqref{eq:TH}  \textbf{depends on the nodal domains of eigenvectors} corresponding to $\lambda_k(\Delta_p)$. 
In this paper, we improve the higher-order Cheeger inequalities in the following sense: we obtain for any forest and for every $p\ge 1$:
\begin{equation}\label{eq:tree-p-C}
\frac{2^{p-1}}{p^p}h_k(G)^p \le \lambda_k(\Delta_p) \le 2^{p-1}h_k(G),    
\end{equation}
where $\lambda_k(\Delta_p)$ is the $k$-th eigenvalue of the graph $p$-Laplacian. The factor $2^{p-1}/p$ in the lower bound does not depend on $k$ and is independent of the number of nodal domains—unlike previous results presented in \eqref{eq:LOT} and \eqref{eq:TH}. Moreover, for general graphs we prove that
\begin{equation}\label{eq:cyclomatic-1-C}
h_{k-\beta}(G) \le \lambda_k(\Delta_1) \le h_k(G),    
\end{equation}
where $\beta$ is the number of independent cycles (cyclomatic number). This shows that the influence of cycles on the Cheeger inequality is limited, and the bounds become tight when the graph is a forest ($\beta=0$).

The proofs of \eqref{eq:tree-p-C} and \eqref{eq:cyclomatic-1-C} make an unexpected use of the following combinatorial Courant-Fischer-Weyl-type minimax representation for $h_k(G)$, which is in fact the first main result of this paper.
\begin{theorem}[\emph{combinatorial} Courant-Fischer-Weyl minimax principle]\label{thm:CombinatorialCourant-Fischer}
For any forest $G=(V,E,\mu,w)$, the Cheeger $k$-constant $h_k(G)$ equals \[\min_{\mathrm{subpartitions~}(A_1,\cdots,A_{k})\mathrm{~of~}V}~~\max_{B\in \mathcal{U}(A_1,\cdots,A_{k})}\varphi(B)=\max_{\mathrm{subpartitions~}(A_1,\cdots,A_{n-k+1})\mathrm{~of~}V}~~\min_{B\in \mathcal{U}(A_1,\cdots,A_{n-k+1})}\varphi(B),\]
where \[\mathcal{U}(A_1,\cdots,A_k)=\{A_{i_1}\cup\cdots\cup A_{i_q}:1\le i_1<\cdots<i_q\le k,\,1\le q\le k\}\] denotes the union-closed family generated by the subpartition $(A_1,\cdots,A_{k})$ of $V$.
\end{theorem}

\begin{example}\label{ex:three-arm-tree}
Consider a tree $G=(V,\,E,\,\mu,\,w)$ with uniform weight, where $V=\{1,\,\cdots,\,10\}$, \[E=\{\{1,\,2\},\,\{2,\,3\},\,\{3,\,4\},\,\{1,\,5\},\,\{5,\,6\},\,\{6,\,7\},\,\{1,\,8\},\,\{8,\,9\},\,\{9,\,10\}\},\] $\mu=w=1$, $k=3$. We can take \[A_1=\{2,\,3,\,4\},\,A_2=\{5,\,6,\,7\},\,A_3=\{8,\,9,\,10\}.\] Then, \[\varphi(A_1)=\varphi(A_2)=\varphi(A_3)=\frac13.\] Conversely, for $\varphi(A)$ to be less than $\frac13$, we must have $|A|>3$, but we only have $|V|=10$ vertices in total, so the pigeonhole principle implies that \[\min_{\mathrm{subpartitions~}(A_1,\cdots,A_{k})\mathrm{~of~}V}~~\max_{B\in \mathcal{U}(A_1,\cdots,A_{k})}\varphi(B)=h_3(G)=\frac13.\] For the other side of the equality, we can take $B_j=\{j+1\}$ for $j=1,\,\cdots,\,8$, then \[\mathcal{U}(B_1,\cdots,B_{n-k+1})=\{2,\,\cdots,\,9\}\] are made of three ``consecutive pieces'', namely $\{2,\,3,\,4\}$, $\{5,\,6,\,7\}$, and $\{8,\,9\}$, so we have \[\min_{B\in \mathcal{U}(B_1,\cdots,B_{n-k+1})}\varphi(B)=\frac13.\] Conversely, by the pigeonhole principle, for any $8$-subset subpartition $(A_1,\,\cdots,\,A_8)$ of $V$, at least on of $\{2,\,3,\,4\}$, $\{5,\,6,\,7\}$, and $\{8,\,9,\,10\}$ is a subset of the union $\mathcal{U}(A_1,\cdots,A_{n-k+1})$, so \[\max_{\mathrm{subpartitions~}(A_1,\cdots,A_{n-k+1})\mathrm{~of~}V}~~\min_{B\in \mathcal{U}(A_1,\cdots,A_{n-k+1})}\varphi(B)=\frac13.\] This illustrates \textbf{Theorem \ref{thm:CombinatorialCourant-Fischer}}.

\begin{figure}[H]
\centering
\begin{tikzpicture}[
vertex/.style={circle,draw,thick,minimum size=0.5cm,inner sep=0.5pt},
omitted/.style={circle,draw=gray,thick,fill=gray!15,minimum size=0.5cm,inner sep=0.5pt,text=gray},
edge/.style={thick},
scale=0.95
]
\node[omitted] (t1) at (0,0) {\scriptsize $1$};
\node[vertex,fill=blue!12] (t2) at (-1,0.8) {\scriptsize $2$};
\node[vertex,fill=blue!12] (t3) at (-2,1.2) {\scriptsize $3$};
\node[vertex,fill=blue!12] (t4) at (-3,1.6) {\scriptsize $4$};
\node[vertex,fill=red!12] (t5) at (-1,-0.8) {\scriptsize $5$};
\node[vertex,fill=red!12] (t6) at (-2,-1.2) {\scriptsize $6$};
\node[vertex,fill=red!12] (t7) at (-3,-1.6) {\scriptsize $7$};
\node[vertex,fill=green!15] (t8) at (1,0) {\scriptsize $8$};
\node[vertex,fill=green!15] (t9) at (2,0) {\scriptsize $9$};
\node[vertex,fill=green!15] (t10) at (3,0) {\scriptsize $10$};
\foreach \i/\j in {1/2,2/3,3/4,1/5,5/6,6/7,1/8,8/9,9/10}
{\draw[edge] (t\i)--(t\j);}
\draw[blue,thick,rounded corners] (-3.35,0.45) rectangle (-0.65,1.95);
\draw[red,thick,rounded corners] (-3.35,-1.95) rectangle (-0.65,-0.45);
\draw[green!50!black,thick,rounded corners] (0.65,-0.35) rectangle (3.35,0.35);
\node[blue] at (-2.15,2.35) {\scriptsize $A_1=\{2,3,4\}$};
\node[red] at (-2.15,-2.35) {\scriptsize $A_2=\{5,6,7\}$};
\node[green!50!black] at (2,0.75) {\scriptsize $A_3=\{8,9,10\}$};

\node[omitted] (b1) at (0,-5) {\scriptsize $1$};
\node[vertex] (b2) at (-1,-4.2) {\scriptsize $2$};
\node[blue] at (-1,-3.7) {\scriptsize $B_1=\{2\}$};
\node[vertex] (b3) at (-2,-3.8) {\scriptsize $3$};
\node[blue] at (-2,-3.3) {\scriptsize $B_2=\{3\}$};
\node[vertex] (b4) at (-3,-3.4) {\scriptsize $4$};
\node[blue] at (-3,-2.9) {\scriptsize $B_3=\{4\}$};
\node[vertex] (b5) at (-1,-5.8) {\scriptsize $5$};
\node[blue] at (-1,-6.3) {\scriptsize $B_4=\{5\}$};
\node[vertex] (b6) at (-2,-6.2) {\scriptsize $6$};
\node[blue] at (-2,-6.7) {\scriptsize $B_5=\{6\}$};
\node[vertex] (b7) at (-3,-6.6) {\scriptsize $7$};
\node[blue] at (-3,-7.1) {\scriptsize $B_6=\{7\}$};
\node[vertex] (b8) at (1,-5) {\scriptsize $8$};
\node[blue] at (1,-4.5) {\scriptsize $B_7=\{8\}$};
\node[vertex] (b9) at (2,-5) {\scriptsize $9$};
\node[blue] at (2,-5.5) {\scriptsize $B_8=\{9\}$};
\node[omitted] (b10) at (3,-5) {\scriptsize $10$};
\foreach \i/\j in {1/2,2/3,3/4,1/5,5/6,6/7,1/8,8/9,9/10}
{\draw[edge] (b\i)--(b\j);}
\foreach \i in {2,3,4,5,6,7,8,9}
{\draw[blue,thick] (b\i) circle (0.31cm);}
\end{tikzpicture}
\caption{The subpartitions used in Example \ref{ex:three-arm-tree}. The top picture gives a three-set subpartition for the min-max side; the bottom picture gives an eight-set subpartition for the max-min side.}
\label{fig:three-arm-tree}
\end{figure}

\end{example}

There are some fundamental‌ equalities in the form of ``min-max = max-min'', such as the von Neumann minimax theorem for convex-concave functions \cite{Rockafellar} and the Courant-Fischer-Weyl minimax principle on the Rayleigh quotient \cite{Bhatia}. 
These capture very important information of the relevant objective function. For example, von Neumann minimax theorem reveals the saddle data of a convex-concave function, while the Courant-Fischer-Weyl theorem establishes minimax representation of eigenvalues of a self-adjoint operator. Since this paper focuses on min-max eigenvalues, we only recall the Courant-Fischer-Weyl minimax principle as follows.

\begin{itemize}
\item For any symmetric matrix $\mathbf{A}\in \mathbb{R}^{n\times n}$, the $k$-th smallest eigenvalue $\lambda_k(\mathbf{A})$ can be written as 
\begin{equation}\label{eq:minmax=maxmin-Courant-eigenval}
\lambda_k(\mathbf{A})=\min_{\substack{\text{linear subspaces }X\text{ of }\mathbb R^n\\\dim X\ge k}}~\max_{x\in X\setminus\{0\}} \frac{x^\top \mathbf{A} x}{x^\top x} = \max_{\substack{\text{linear subspaces }X\text{ of }\mathbb R^n\\\dim X\ge n-k+1}}~\min_{x\in X\setminus\{0\}} \frac{x^\top \mathbf{A} x}{x^\top x}.
\end{equation}
\end{itemize}

It is worth noting that von Neumann minimax theorem involves duality which is commonly used in optimization, and thus many different analogues and generalizations have been obtained, e.g., there is a famous combinatorial version called the max-flow min-cut theorem. However, to the best of our knowledge, no combinatorial analogues of the Courant-Fischer-Weyl minimax principle had been proposed prior to our paper. 
Actually, if we focus solely on the ``min-max = max-min'' equality 
\begin{equation}\label{eq:minmax=maxmin-Courant}
\min_{\substack{\text{linear subspaces }X\text{ of }\mathbb R^n\\\dim X= k}}~\max_{x\in X\setminus\{0\}} \frac{x^\top \mathbf{A} x}{x^\top x} = \max_{\substack{\text{linear subspaces }X\text{ of }\mathbb R^n\\\dim X= n-k+1}}~\min_{x\in X\setminus\{0\}} \frac{x^\top \mathbf{A} x}{x^\top x},   
\end{equation}
it would be surprising if \eqref{eq:minmax=maxmin-Courant} has no generalizations. 

Theorem \ref{thm:CombinatorialCourant-Fischer} gives the first combinatorial analogue of the Courant–Fischer–Weyl principle for $h_k(G)$. 
This form can be viewed as a combinatorial analog of the minimax principle on real symmetric matrices. 
It is known that \eqref{eq:minmax=maxmin-Courant} implies the Courant-Fischer-Weyl minimax principle for graph Laplacian:
\begin{equation}\label{eq:classical-minmax-Lap}
\min_{\substack{\text{linear subspaces }X\text{ of }\mathbb R^n\\\dim X= k}}~\max_{x\in X\setminus\{0\}} \frac{\sum\limits_{u\sim v}w_{uv}(x_u-x_v)^2}{\sum\limits_v\mu_vx_v^2} = \max_{\substack{\text{linear subspaces }X\text{ of }\mathbb R^n\\\dim X= n-k+1}}~\min_{x\in X\setminus\{0\}} \frac{\sum\limits_{u\sim v}w_{uv}(x_u-x_v)^2}{\sum\limits_v\mu_vx_v^2},
\end{equation}
where we write $u\sim v$ to mean $\{u,v\}$ is an edge. This minimax principle characterizes the $k$-th eigenvalue  $\lambda_k(\Delta_2)$ of the linear Laplacian on graphs. We also propose a nonlinear analogue of \eqref{eq:classical-minmax-Lap} which involves the $\ell^1$-Rayleigh quotient of the graph 1-Laplacian:

\begin{theorem}[\emph{nonlinear} Courant-Fischer-Weyl minimax principle]\label{thm:nonlinearCourant-Fischer}
For any forest, 
\begin{align*}\min_{\substack{\mathrm{linear~subspaces~}X\mathrm{~of~}\mathbb R^n\\\dim X= k}}~\max_{x\in X\setminus\{0\}} \frac{\sum\limits_{u\sim v}w_{uv}|x_u-x_v|}{\sum\limits_v\mu_v|x_v|} = \max_{\substack{\mathrm{linear~subspaces~}X\mathrm{~of~}\mathbb R^n\\\dim X= n-k+1}}~\min_{x\in X\setminus\{0\}} \frac{\sum\limits_{u\sim v}w_{uv}|x_u-x_v|}{\sum\limits_v\mu_v|x_v|}
\end{align*}
\end{theorem}

\emph{It is noteworthy that both equalities in Theorem \ref{thm:CombinatorialCourant-Fischer} and Theorem \ref{thm:nonlinearCourant-Fischer} are equal to $h_k(G)$ when $G$ is a forest.} 

These results are fascinating because equalities of the ``min-max = max-min'' form rarely arise in both nonlinear and combinatorial situations. For example, in the field of nonlinear eigenvalue problems, there are many conjectures and open problems on ``min-max = max-min''  equalities \cite{Diening}. Specifically, for the $p$-Laplacian, there are several different definitions of $k$-th variational eigenvalue (using the Krasnoselskii genus, the Conner–Floyd index, or the Yang index), where $k\in\mathbb{Z}_+$. It is a long-standing open problem whether these definitions coincide. 
For forests and the case $p=1$, we prove that they are all equal, and moreover they all equal the Cheeger $k$-constant. Surprisingly, our solution is still based on Theorems \ref{thm:CombinatorialCourant-Fischer} and  \ref{thm:nonlinearCourant-Fischer}. 

This paper is structured as follows. In Section \ref{sec:main}, we present our main results, including the combinatorial and nonlinear minimax principles (Theorems \ref{thm:CombinatorialCourant-Fischer} and \ref{thm:nonlinearCourant-Fischer}), the refined higher order Cheeger inequality \eqref{eq:tree-p-C} and the equality $h_k(G)=\lambda_k(\Delta_1)$ for forests, as well as the refined Cheeger inequalities with cycle corrections \eqref{eq:cyclomatic-1-C}. 
Section \ref{sec:proof} contains all the proofs, organized into a logical flow diagram that highlights how the combinatorial lemmas lead to the analytical applications. We also recall known monotonicity results for $p$-Laplacian eigenvalues that allow us to extend the $p=1$ case to all $p\ge 1$.

\section{Main Results}
\label{sec:main}
To express our results more concisely, we shall adopt some notation and conventions. 

Let \[\mathcal{P}_k(V):=\{(A_1,\cdots,A_{k}):~\varnothing\ne A_i\subseteq V,\,A_i\cap A_j=\varnothing\,\forall i\ne j\}\] be the set of all subpartitions of $V$ with $k$ subsets. 

We have the following elementary observation:
\begin{proposition}\label{prop:unioninv}
For any weighted graph $G=(V,E,\mu,w)$,
$$h_k(G):=\min_{(A_1,\cdots,A_{k})\in \mathcal{P}_k(V)}~\max_{1\le i\le k}\varphi(A_i)=\min_{(A_1,\cdots,A_{k})\in \mathcal{P}_k(V)}~\max_{B\in \mathcal{U}(A_1,\cdots,A_{k})}\varphi(B)$$
where \[\mathcal{U}(A_1,\cdots,A_k):=\{A_{i_1}\cup\cdots\cup A_{i_q}:1\le i_1<\cdots<i_q\le k,\,1\le q\le k\}\] stands for the union closed family generated by the $k$-tuple of sets $(A_1,\cdots,A_k)$. 
\end{proposition}
For any weighted graph $G=(V,E,\mu,w)$ with $n$ vertices, let 
$$\ell_k(G):=\max_{(A_1,\cdots,A_{n-k+1})\in \mathcal{P}_{n-k+1}(V)}~\min_{B\in \mathcal{U}(A_1,\cdots,A_{n-k+1})}\varphi(B)$$
be the $k$-th max-min Cheeger constant, and let $$\underline{\ell_k}(G) :=\max_{A\subseteq V,~|A|=n-k+1}h(A)$$ be the $k$-th Dirichlet Cheeger constant \cite{DTZ}, where \[h(A):=\min_{B\subseteq A}\varphi(B)\] 
denotes the Cheeger constant of $A$ with respect to $G$. 


We adopt a novel way, namely, by removing vertices from a subpartition of some suitably chosen subclass of subpartitions minimizing the maximum $\varphi$-value, to prove our following main result of an alternate max-min characterization of the Cheeger $k$-constant of weighted forests:
\begin{theorem}\label{thm:h_k=H_k}
For any weighted graph $\tilde G=(\tilde V,\tilde E,\tilde\mu,\tilde w)$ with $n$ vertices, for any $k\in\{1,2,\cdots,n\}$, we have \[h_k(\tilde G)\ge \ell_k(\tilde G)\ge \underline{\ell_k}(\tilde G).\] 
For any weighted forest $G=(V,E,\mu,w)$ with $n$ vertices, we must have \[h_k(G)= \ell_k(G)= \underline{\ell_k}(G).\]
\end{theorem}
It is evident that the equality $h_k(G)= \ell_k(G)$ concerning forests in the aforementioned theorem is precisely the combinatorial Courant-Fischer-Weyl minimax principle we described in the introduction (see Theorem \ref{thm:CombinatorialCourant-Fischer}).

We present a generalization of {Theorem \ref{thm:h_k=H_k}} as follows.
\begin{theorem}\label{thm:cycle-number-inequ}
For any weighted graph $\tilde G=(\tilde V,\tilde E,\tilde\mu,\tilde w)$,  \[h_k(\tilde G)\ge \ell_k(\tilde G)\ge \underline{\ell_k}(\tilde G)\ge h_{k-\beta}(\tilde G),\] where $\beta:=|\tilde E|-|\tilde V|+c$ is the total number of independent cycles of $\tilde G$, and $c$ is the number of connected components of $\tilde G$. 
\end{theorem}

\subsection{Applications on min-max eigenvalues of Graph 1-Laplacian of Forests}\label{sec:minmax-eigen-1-Lap}
In the study of the $p$-Laplacian eigenvalue problems \cite{BH25,Burger,Drutu,Fazeny,HN25,Keller,LMR,Weihs1,Weihs2}, at least three sequences of eigenvalues have been introduced over the years. All of them are obtained by a \emph{minimax} procedure:
\begin{equation}\label{eq:continuous}
\lambda_k^{\mathrm{ind}}(\Delta_p):=\inf_{\substack{S\text{ origin-symmetric, compact}\\\mathrm{ind}(S)\ge k}}~\sup_{f\in S}\frac{\int_\Omega|\nabla f(z)|^pdz}{\int_\Omega|f(z)|^pdz}
\end{equation}
where $\mathrm{ind}(\cdot)$ is an admissible index for origin-symmetric compact subsets of the Sobolev space $H_0^1(\Omega)$. There are three admissible indexes commonly used for $p$-Laplacian:
\begin{itemize}
\item Krasnoselskii genus: The \emph{Krasnoselskii genus} of an origin-symmetric compact set $S$ is defined by $$\gamma(S):=\min\{k\in\mathbb{Z}^+:\exists\text{ odd continuous }\eta:S\to \mathbb{R}^k\setminus\{0\}\}.$$ 
See \cite{Coffman} for details.
\item  Conner-Floyd index: The \emph{Conner-Floyd index} $\gamma^+$ of an origin-symmetric compact set $S$ is defined by $$\gamma^+(S):=\min\{k\in\mathbb{Z}^+:\exists\text{ odd continuous }\eta:\mathbb{S}^{k-1}\to S\setminus\{0\}\},$$ 
where $\mathbb{S}^{k-1}$ stands for the unit sphere of $\mathbb{R}^k$. 
The min-max eigenvalues of $p$-Laplacian using $\mathrm{ind}=\gamma^+$ are called the Drabek-Robinson eigenvalues \cite{Drabek}.
\item Yang index: This index, denoted by $\textrm{Y-ind}(\cdot)$, is defined via homology information \cite{Yang}. We will not write down the definition explicitly, but interested readers may refer to \cite{Perera}.
\end{itemize}

The above three indices possess some common properties, which prompts us to introduce the following definition.

\begin{defn}[admissible index]
Let $\mathcal{S}$ be the set of all origin-symmetric compact subsets. 
An index $\mathrm{ind}:\mathcal{S}\to \mathbb{Z}_{\ge0}$ is \emph{admissible} if it satisfies the following properties:
\begin{itemize}
\item[(1)] 
$\mathrm{ind}(\mathbb{S}^{k-1})=k$, where $k=1,2,\cdots$
\item[(2)] monotonicity: If $A\subseteq A'$, then $$\mathrm{ind}(S)\le \mathrm{ind}(S').$$
\item[(3)] continuity: For any $S$, there exists a closed neighborhood $U$ of $S$ such that $$\mathrm{ind}(U)= \mathrm{ind}(S).$$
\item[(4)] nondecreasing under odd continuous map: For any odd continuous map $\eta$, $$\mathrm{ind}(S)\le  \mathrm{ind}(\eta(S)).$$
\end{itemize}
\end{defn}

By the standard approach in defining variational min-max eigenvalues of $\Delta_p$, we state that for every $k$, $\lambda_k^{\mathrm{ind}}(\Delta_p)$ defined with an admissible index must be an eigenvalue of $\Delta_p$ (see \cite{DTZ,Zhang}). 

It is known that $\lambda_1^{\gamma}(\Delta_p)= \lambda_1^{\gamma^+}(\Delta_p)$, $\lambda_2^{\gamma}(\Delta_p)= \lambda_2^{\gamma^+}(\Delta_p)$ and $\lambda_k^{\gamma}(\Delta_p)\le \lambda_k^{\gamma^+}(\Delta_p)$ for general $k$. However, the fundamental question of whether the general equality $\lambda_k^{\gamma}(\Delta_p)= \lambda_k^{\gamma^+}(\Delta_p)$ holds for every $k$ remains unresolved. For details, we refer to the open problem proposed by Luigi De Pascale in the 2013 Oberwolfach workshop \cite{Diening}.

In discrete settings, similar problems remain unsolved for the graph $p$-Laplacian. Let 
\begin{equation}\label{eq:discrete}
\lambda_k^{\mathrm{ind}}(\Delta_p):=\inf_{\substack{S\text{ origin-symmetric, compact}\\\mathrm{ind}(S)\ge k}}~\max_{x\in S}\frac{\sum\limits_{u\sim v}w_{uv}|x_u-x_v|^p}{\sum\limits_v\mu_v|x_v|^p} 
\end{equation}
be the $k$-th min-max eigenvalue of $\Delta_p$ on a graph $G=(V,E,\mu,w)$, and let
\begin{equation}\label{eq:discrete/maxmin}
\underline{\lambda_k^{\mathrm{ind}}}(\Delta_p):=\sup_{\substack{S\text{ origin-symmetric, compact}\\\mathrm{ind}(S)\ge n-k+1}}~\min_{x\in S}\frac{\sum\limits_{u\sim v}w_{uv}|x_u-x_v|^p}{\sum\limits_v\mu_v|x_v|^p} 
\end{equation}
be the $k$-th max-min eigenvalue of $\Delta_p$ on $G=(V,E,\mu,w)$ with $n$ vertices. 

As a direct corollary of Theorem \ref{thm:h_k=H_k}, we can give an affirmative answer to this question for 1-Laplacian of forests:

\begin{theorem}\label{thm:1-Lap=h}
Let $\Delta_1$ denote the 1-Laplacian of a forest $G$ with $n$ vertices. Then, for any admissible index, for any $k$, we have the equality
\begin{align*}
h_k(G)=\lambda_k^{\mathrm{ind}}(\Delta_1)=\underline{\lambda_k^{\mathrm{ind}}}(\Delta_1).
\end{align*}
\end{theorem}

In particular, we have
\begin{cor}\label{cor:minmax-indepen-index}
Let $\Delta_1$ denote the 1-Laplacian of a forest. Then, for any $k$, 
$$ \lambda_k^{\gamma}(\Delta_1)=\lambda_k^{\gamma^+}(\Delta_1) =\lambda_k^{\textrm{Y-ind}}(\Delta_1).$$
\end{cor}

\begin{cor}[
\cite{Deidda-thesis,Deidda}]\label{cor:tree-1-Lap-h}
For any forest $G$, we have $$\lambda_{k}^\gamma(\Delta_1)= h_k(G).$$
\end{cor}

From Theorem \ref{thm:1-Lap=h}, we realize that Theorem \ref{thm:h_k=H_k} actually induces a combinatorial proof of Corollary \ref{cor:tree-1-Lap-h}. 

\subsection{Applications on Refined Multi-way Cheeger Inequalities
}\label{sec:refined-Cheeger}

In this paper, we prove Theorems \ref{thm:h_k=H_k} and 
\ref{thm:cycle-number-inequ}, and combining with the known properties presented in Section \ref{sec:tools}, we improve the multi-way Cheeger inequality for forests and on graphs with a small total number of independent cycles. We shall omit the superscript $\gamma$ in $\lambda_{k}^\gamma(\Delta_p)$ and write it as $\lambda_{k}(\Delta_p)$ instead. This is because most of the literature on graph $p$-Laplacians adopts this notation \cite{Deidda,LMR,Tudisco}. Furthermore, we assume that the vertex weights and the edge weights satisfy $\mu_v=\sum_{u\in V:u\sim v}w_{uv}$. Such an assumption is commonly employed because of the advantage that the factors in Cheeger's inequality are independent of the choice of weights.

Together with Corollary \ref{cor:tree-1-Lap-h} and a monotonicity property of the eigenvalues of graph $p$-Laplacian (c.f. \cite{Zhang}), we are able to establish the following multi-way  Cheeger inequality:

\begin{cor}\label{cor:tree-p-Lap}
For any weighted forest $G$ satisfying $\mu_v=\sum_{u\in V:u\sim v}w_{uv}$, $\forall v\in V$, for any $p\ge 1$, we have
$$\frac{2^{p-1}}{p^p}h_k(G)^p\le \lambda_{k}(\Delta_p)\le 2^{p-1} h_k(G).$$
\end{cor}

Corollary \ref{cor:tree-p-Lap} not only includes Corollary \ref{cor:tree-1-Lap-h} as a special case, but also refines the classical Cheeger inequality on trees by Miclo. 
We can rewrite the inequality in  {Corollary \ref{cor:tree-p-Lap}} as
$$
\swarrow\frac{1}{2^{p-1}}\lambda_k(\Delta_p)\le h_k(G)\le \frac p2 (2\lambda_k(\Delta_p))^{\frac1p}\nearrow.
$$
Then, by the increasing property of $\frac p2 (2\lambda_k(\Delta_p))^{\frac1p}$ with respect to $p\in[1,+\infty)$, and the decreasing property of $\frac{1}{2^{p-1}}\lambda_k(\Delta_p)$ with respect to $p\in[1,+\infty)$, the inequality becomes tight when $p$ tends to 1. 

Thanks to  {Theorem \ref{thm:cycle-number-inequ}}, {Corollary \ref{cor:tree-1-Lap-h}} can be generalized to apply to any graph with a small cyclomatic number $\beta$ (number of independent cycles).
\begin{theorem}\label{thm:beta}
For any weighted graph $G=(V,E,\mu,w)$, 
$$h_{k-\beta}(G) \le \lambda_{k}(\Delta_1)\le h_k(G),$$
where $\beta:=|E|-|V|+1$. 
\end{theorem}

Similarly, we can generalize {Corollary \ref{cor:tree-p-Lap}} to the following formulation:
\begin{theorem}\label{thm:beta-p}
For any weighted graph $G=(V,E,\mu,w)$ satisfying $\mu_v=\sum_{u\in V:u\sim v}w_{uv}$, $\forall v\in V$, for any $p\ge 1$, we have 
$$\frac{2^{p-1}}{p^p}h_{k-\beta}(G)^p\le \lambda_{k}(\Delta_p)\le 2^{p-1} h_k(G).$$
\end{theorem}

Note that the subscript $k-\beta$ of the term $h_{k-\beta}$ does not depend on the nodal domain counts, and the factor $\frac{2^{p-1}}{p^p}$ in the lower bound is independent of both $k$ and graphs. This improves the main theorem in \cite{Tudisco}. 

A \emph{unicyclic graph} is a graph that has exactly one circuit. 
\begin{cor}\label{cor:unicyclic}
For any weighted unicyclic graph $G=(V,E,\mu,w)$ satisfying $\mu_v=\sum_{u\in V:u\sim v}w_{uv}$, $\forall v\in V$, we have $\frac{2^{p-1}}{p^p}h_{k-1}(G)\le \lambda_{k}(\Delta_p)\le 2^{p-1} h_k(G)$.   
\end{cor}

\section{Proofs}
\label{sec:proof}
This section is devoted to providing detailed proofs of all the theorems presented in this paper. Figure \ref{fig:placeholder} describes the logic flow between our results, rectangles indicate combinatorial results, ellipses indicate analytical results, white indicates auxiliary and less important results, grey indicates main results.

\tikzstyle{startstop} = [ellipse, rounded corners, minimum width=1cm, minimum height=1cm,text centered, draw=black, fill=white]
\tikzstyle{startstop2} = [ellipse, minimum width=1cm, minimum height=1cm,text centered, draw=black, fill=gray!10]
\tikzstyle{startstop3} = [rectangle, minimum width=1cm, minimum height=1cm,text centered, draw=black, fill=white]
\tikzstyle{startstop4} = [rectangle, minimum width=1cm, minimum height=1cm,text centered, draw=black, fill=gray!10]

\begin{figure}
    \centering
\begin{tikzpicture}[scale=1,node distance=3.9cm]
\node(L1)[startstop3]{Lemma \ref{lemma:basic-phi} };
\node(L2)[startstop3,below of=L1,yshift=1.5cm,xshift=-3.5cm]{Lemma \ref{lem:cheeger-set-disjoint}};  
\node(L3)[startstop3,below of=L1,yshift=2.6cm,xshift=0cm]{Lemma \ref{lem:pigeonhole}};  
\node(P1)[startstop3,below of=L1,yshift=-0.5cm,xshift=0cm]{Proposition \ref{prop:unioninv}};
\node(T3)[startstop4,right of=L2,yshift=0cm,xshift=3cm]{Theorem \ref{thm:h_k=H_k}};
\node(T4)[startstop4,below of=T3,yshift=-8cm,xshift=0cm]{Theorem \ref{thm:cycle-number-inequ}};
\node(T6)[startstop2,right of=T4,yshift=4cm,xshift=0cm]{Theorem \ref{thm:beta}};
\node(T1)[startstop4,below of=T3,yshift=-1.5cm,xshift=-1.5cm]{Theorem \ref{thm:CombinatorialCourant-Fischer}};
\node(T5)[startstop2,right of=T3,yshift=6cm,xshift=-3.8cm]{Theorem \ref{thm:1-Lap=h}};
\node(C2)[startstop,above of=T5,yshift=-2cm,xshift=4cm]{Corollary \ref{cor:tree-1-Lap-h}};
\node(C1)[startstop,left of=T5,yshift=0cm,xshift=0cm]{Corollary \ref{cor:minmax-indepen-index}};
\node(T2)[startstop2,right of=T3,yshift=2cm,xshift=-2cm]{Theorem \ref{thm:nonlinearCourant-Fischer}};
\node(L4)[startstop,above of=T5,yshift=-8cm,xshift=8cm]{Lemma \ref{lem:p-monotonicity}};
\node(C3)[startstop,above of=L4,yshift=0cm,xshift=0cm]{Corollary \ref{cor:tree-p-Lap}};
\node(L7)[startstop,below of=T5,yshift=2cm,xshift=3.5cm]{Lemma \ref{lem:general-graph-delta}};
\node(L6)[startstop,below of=L7,yshift=0cm,xshift=2cm]{Lemma \ref{lem:levelset}};
\node(L5)[startstop,above of=L7,yshift=-2cm,xshift=0cm]{Lemma \ref{lem:intersect}};
\node(T7)[startstop2,below of=L4,yshift=0cm,xshift=0cm]{Theorem \ref{thm:beta-p}};
\node(C4)[startstop,below of=T7,yshift=1cm,xshift=0cm]{Corollary \ref{cor:unicyclic}};
\draw[decoration={markings,mark=at position 1 with{\arrow[scale=2,>=stealth]{>}}},postaction={decorate}](P1) to
node[anchor=south]{ } (T3);
\draw[decoration={markings,mark=at position 1 with{\arrow[scale=2,>=stealth]{>}}},postaction={decorate}](L1) to
node[anchor=north] { }(T3);
\draw[decoration={markings,mark=at position 1 with{\arrow[scale=2,>=stealth]{>}}},postaction={decorate}](L2) to
node[anchor=north] { }(T3);
\draw[decoration={markings,mark=at position 1 with{\arrow[scale=2,>=stealth]{>}}},postaction={decorate}](L2) to
node[anchor=north] { }(P1);
\draw[decoration={markings,mark=at position 1 with{\arrow[scale=2,>=stealth]{>}}},postaction={decorate}](L3) to
node[anchor=north] { }(T3);
\draw[decoration={markings,mark=at position 1 with{\arrow[scale=2,>=stealth]{>}}},postaction={decorate}](L1) to
node[anchor=north] { }(L2);
\draw[decoration={markings,mark=at position 1 with{\arrow[scale=2,>=stealth]{>}}},postaction={decorate}](T3) to
node[anchor=north]{ } (T5);
\draw[decoration={markings,mark=at position 1 with{\arrow[scale=2,>=stealth]{>}}},postaction={decorate}](T3) to
node[anchor=north]{ } (T4);
\draw[decoration={markings,mark=at position 1 with{\arrow[scale=2,>=stealth]{>}}},postaction={decorate}](T3) to
node[anchor=north]{ } (T1);
\draw[decoration={markings,mark=at position 1 with{\arrow[scale=2,>=stealth]{>}}},postaction={decorate}](T4)to node[anchor=north]{ }(T6);
\draw[decoration={markings,mark=at position 1 with{\arrow[scale=2,>=stealth]{>}}},postaction={decorate}](T3) to node[anchor=north]{ }(T6);
\draw[decoration={markings,mark=at position 1 with{\arrow[scale=2,>=stealth]{>}}},postaction={decorate}](L7) to node[anchor=north]{ }(T6);
\draw[decoration={markings,mark=at position 1 with{\arrow[scale=2,>=stealth]{>}}},postaction={decorate}](L7) to node[anchor=north]{ }(T2);
\draw[decoration={markings,mark=at position 1 with{\arrow[scale=2,>=stealth]{>}}},postaction={decorate}](T3) to node[anchor=north]{ }(T2);
\draw[decoration={markings,mark=at position 1 with{\arrow[scale=2,>=stealth]{>}}},postaction={decorate}](T6) to node[anchor=north]{ }(T7);
\draw[decoration={markings,mark=at position 1 with{\arrow[scale=2,>=stealth]{>}}},postaction={decorate}](L4) to node[anchor=north]{ }(T7);
\draw[decoration={markings,mark=at position 1 with{\arrow[scale=2,>=stealth]{>}}},postaction={decorate}](L7) to node[anchor=north]{ }(T5);
\draw[decoration={markings,mark=at position 1 with{\arrow[scale=2,>=stealth]{>}}},postaction={decorate}](L5) to node[anchor=north]{ }(L7);
\draw[decoration={markings,mark=at position 1 with{\arrow[scale=2,>=stealth]{>}}},postaction={decorate}](L6) to node[anchor=north]{ }(L7);
\draw[decoration={markings,mark=at position 1 with{\arrow[scale=2,>=stealth]{>}}},postaction={decorate}](P1) to node[anchor=north]{ }(T1);
\draw[decoration={markings,mark=at position 1 with{\arrow[scale=2,>=stealth]{>}}},postaction={decorate}](T5) to node[anchor=north]{ }(C1);
\draw[decoration={markings,mark=at position 1 with{\arrow[scale=2,>=stealth]{>}}},postaction={decorate}](T5) to node[anchor=north]{ }(C2);
\draw[decoration={markings,mark=at position 1 with{\arrow[scale=2,>=stealth]{>}}},postaction={decorate}](C2) to node[anchor=north]{ }(C3);
\draw[decoration={markings,mark=at position 1 with{\arrow[scale=2,>=stealth]{>}}},postaction={decorate}](L4) to node[anchor=north]{ }(C3);
\draw[decoration={markings,mark=at position 1 with{\arrow[scale=2,>=stealth]{>}}},postaction={decorate}](L2) to node[anchor=north]{ }(T4);
\draw[decoration={markings,mark=at position 1 with{\arrow[scale=2,>=stealth]{>}}},postaction={decorate}](T7) to node[anchor=north]{ }(C4);
\end{tikzpicture}
    \caption{Logic Flow}
    \label{fig:placeholder}
\end{figure}

\subsection{Proofs of the Main Theorems (Combinatorial)}
For any two nonempty disjoint subsets $A$ and $B$ of $V$, define \[E(A,B):=\{\{u,v\}\in E:u\in A,v\in B\}\] to be the set of edges with one vertex in $A$ and the other in $B$. Two nonempty disjoint subsets $A$ and $B$ are said to be \emph{nonadjacent} if $E(A,B)=\varnothing$.

\begin{lemma}\label{lemma:basic-phi}
For any disjoint subsets $A$ and $B$ of $V$, we have $\varphi(A\cup B)\le \max\left\{\varphi(A),\varphi(B)\right\}$. If we further assume that $E(A,B)=\varnothing$, then we have $$\min\left\{\varphi(A),\varphi(B)\right\} \le \varphi(A\cup B)\le \max\left\{\varphi(A),\varphi(B)\right\}.$$
\end{lemma}

\begin{proof}[Proof of Lemma \ref{lemma:basic-phi}]
Let $a=w(\partial A)$, $b=\mu(A)$, $c=w(\partial B)$, $d=\mu(B)$. Then \[w(\partial(A\cup B))\le a+c\text{ and }\mu(A\cup B)=b+d,\] so $$\varphi(A\cup B)\le\frac{a+c}{b+d}\le\max\Big\{\frac{a}{b},\frac{c}{d}\Big\}=\max\left\{\varphi(A),\varphi(B)\right\}.$$ 
If in addition $E(A,B)=\varnothing$, then $w(\partial(A\cup B))=a+c$ and $$\varphi(A\cup B)=\frac{a+c}{b+d}\ge\min\Big\{\frac{a}{b},\frac{c}{d}\Big\}=\min\left\{\varphi(A),\varphi(B)\right\}.$$
\end{proof}

The next statement is a consequence of the preceding lemma.

\begin{lemma}\label{lem:cheeger-set-disjoint}
For any subpartition $(B_1,\cdots,B_k)$ of $V$, we have \[h(B_1\cup\cdots\cup B_k)\le \min\{h(B_1),\cdots,h(B_k)\}.\] If we further assume that $E(B_i,B_j)=\varnothing$ whenever $i\ne j$, then we have \[h(B_1\cup\cdots\cup B_k)=\min\{h(B_1),\cdots,h(B_k)\}.\]
\end{lemma}

\begin{proof}[Proof of Lemma \ref{lem:cheeger-set-disjoint}]
The first assertion is due to the fact that $h(B)\le h(A)$ whenever $A\subset B$. 

For the second assertion, suppose $C\subseteq B_1\cup\cdots\cup B_k$ is such that $\varphi(C)=h(B_1\cup\cdots\cup B_k)$, then we can assume that $C$ is connected, because otherwise we can replace $C$ by one of its connected components that possesses the minimum $\varphi$-value without affecting the minimality by Lemma \ref{lemma:basic-phi}, so $C\subseteq B_j$ for some $1\le j\le k$. Thus, \[\varphi(C)=h(B_j)=\min\{h(B_1),\cdots,h(B_k)\},\] this concludes the proof.
\end{proof}

We note that the geometric versions of Lemmas \ref{lemma:basic-phi} and \ref{lem:cheeger-set-disjoint} are known \cite{Bobkov-Parini,Mazon}. 

\begin{lemma}\label{lem:pigeonhole}
Let $n$ be a positive integer. Suppose that \[(A_1,\cdots,A_k)\text{ and }(B_1,\cdots,B_{n-k+1})\] are subpartitions of $[n]:=\{1,\cdots, n\}$, then there exists a nonempty subset $\{i_1,\cdots,i_t\}$ of $[k]$ and a nonempty subset $\{j_1,\cdots,j_s\}$ of $[n-k+1]$ such that \[A_{i_1}\cup\cdots\cup A_{i_t}=B_{j_1}\cup\cdots\cup B_{j_s}.\]
\end{lemma}

\begin{proof}[Proof of Lemma \ref{lem:pigeonhole}]
We apply strong induction on $n$. The base case $n=1$ is obvious. In general, if $(A_1,\cdots,A_k)$ and $(B_1,\cdots,B_{n-k+1})$ are partitions of $[n]$, then we are done. Otherwise, let \[C=\left(\bigcup_{1\le i\le k} A_i\right)\bigcap\left(\bigcup_{1\le j\le n-k+1} B_j\right),\] then we have $c:=|C|<n$. Note that at most $n-c$ subsets among the $A_i$'s and the $B_j$'s contain elements outside $C$. In other words, at least $c+1$ subsets contain only elements in $C$. Now we apply the induction hypothesis on $C$ and the two subpartitions formed by those at least $c+1$ subsets.
\end{proof}

\subsubsection{Proof of Proposition \ref{prop:unioninv}}
\begin{proof}
The proposition follows immediately from Lemma \ref{lem:cheeger-set-disjoint}, which implies that choosing from $\mathcal U(A_1,\cdots,A_k)$ will not make it strictly larger than choosing from $\{A_1,\cdots,A_k\}$.
\end{proof}

\subsubsection{Proof of Theorem \ref{thm:h_k=H_k}}
\begin{proof}
Due to its overwhelming length, we divide our proof into several steps.

\vspace{0.3cm}

\noindent\underline{Step 1. Setup.}
\newline\noindent For ease of presentation, we prove the case when $G$ is a weighted tree, whereas the proof for the general case of weighted forests is done by mimicking the following proof. Let \[S_1:=\argmin_{\mathcal A=(A_1,\cdots,A_k)\in\mathcal{P}_k(V)}~\max_{1\le i\le k}\varphi(A_i)\] be the set of all subpartitions that minimize the maximum $\varphi$-value among the subsets. For reasons that will become clear later, we do not attempt to start working on a subpartition in $S_1$. Instead, we need more requirements on the subpartitions than that and further refine our desired class by letting \[S_2:=\argmin_{\mathcal A=(A_1,\cdots,A_k)\in S_1}f(\mathcal A),\] where \[f(\mathcal A)=|\{1\le i\le k\colon\varphi(A_i)=h_k\}|,\] be the set of all subpartitions in $S_1$ minimizing the number of subsets sharing the maximum $\varphi$-value and \[S_3:=\argmin_{\mathcal A=(A_1,\cdots,A_k)\in S_2}~\mu(\mathcal A)\] be the set of all subpartitions in $S_2$ that minimize the total weighted degree, where $$\mu(\mathcal A):=\sum_{i=1}^k\mu(A_i):=\sum_{v\in A_1\cup\cdots\cup A_k}\mu_v.$$ Now, pick any subpartition $\mathcal A=(A_1,\cdots,A_k)\in S_3$. By reordering if necessary, suppose that \[\varphi(A_1)\le\cdots\le\varphi(A_k)=h_k\] without loss of generality.

\vspace{0.3cm}

\noindent\underline{Step 2. Achieving pairwise nonadjacency.}
\newline\noindent We prove that there exist $v_1\in A_1,\cdots,v_{k-1}\in A_{k-1}$ such that \[A_1\setminus\{v_1\},\cdots,A_{k-1}\setminus\{v_{k-1}\},V\setminus(A_1\cup\cdots\cup A_{k-1})\] are pairwise nonadjacent in this step.

Indeed, since $\mathcal A=(A_1,\cdots,A_k)\in S_3$, each $A_i$ is connected. (This is because otherwise, if some $A_j$ is not connected, then we write \[A_j=A_{j1}\cup\cdots\cup A_{jm},\] where $A_{j1},\cdots,A_{jm}$ are the $m$ connected components of $A_j$. Now, by Lemma \ref{lemma:basic-phi}, since \[E(A_{j1},A_{j2}\cup\cdots\cup A_{jm})=\varnothing,\] we have \[\min\{\varphi(A_{j1}),\varphi(A_{j2}\cup\cdots\cup A_{jm})\}\le\varphi(A_j).\] If it happens that $\varphi(A_{j1})>\varphi(A_j)$, then we apply Lemma \ref{lemma:basic-phi} and the same procedure on \[A_{j2}\cup\cdots\cup A_{jm}=A_{j2}\cup(A_{j3}\cup\cdots\cup A_{jm}),\] and eventually we will obtain some $1\le p\le m$ such that $\varphi(A_{jp})\le\varphi(A_j)$. Now note that the subpartition \[\tilde{\mathcal A}=(A_1,\cdots,A_{jp},\cdots,A_k)\] is in $S_1$, since replacing by $A_j$ by $A_{jp}$ does not increase the maximum $\varphi$-value. For this same reason, we can infer that $\tilde{\mathcal A}\in S_2$ as well. However, the total weighted degree of $\tilde{\mathcal A}$ is smaller than that of $\mathcal A$, which contradicts the assumption that $\mathcal A\in S_3$. Therefore, each $A_i$ is connected.) 

Now, recall that the graph $G$ we are working on is assumed to be a tree, so we can make the $k$ connected subsets $A_1,\cdots,A_k$ pairwise nonadjacent by removing a vertex in each of some $k-1$ subsets. (See Figure \ref{fig:pairwisenonadj}).

\begin{figure}
\centering
\begin{tikzpicture}[scale=2]
\draw (0,0)--(-2,-1)--(-3,-2);
\draw (0,0)--(-1,-1);
\draw (0,0)--(1,-1);
\draw (0,0)--(2,-1)--(3,-2);
\draw (2,-1)--(1.7,-1.4);
\draw (2,-1)--(2.3,-1.4);
\draw (-2,-1)--(-2.3,-1.4);
\draw (-2,-1)--(-1.7,-1.4);
\draw (-2,-1)--(-1,-2);
\draw (2,-1)--(1,-2);
\draw (-3,-2)--(-3.25,-3);
\draw (-3,-2)--(-2.75,-3);
\draw (-1,-2)--(-1.25,-3);
\draw (-1,-2)--(-0.75,-3);
\draw (1,-2)--(0.75,-3);
\draw (1,-2)--(1.25,-3);
\draw (3,-2)--(2.75,-3);
\draw (3,-2)--(3.25,-3);
\draw (-3.75,-4)--(-3.75,-4.5);
\draw (-3.25,-4)--(-3.25,-4.5);
\draw (-2.75,-4)--(-2.75,-4.5);
\draw (-2.25,-4)--(-2.25,-4.5);
\draw (-1.75,-4)--(-1.75,-4.5);
\draw (-1.25,-4)--(-1.25,-4.5);
\draw (-0.75,-4)--(-0.75,-4.5);
\draw (-0.25,-4)--(-0.25,-4.5);
\draw (0.25,-4)--(0.25,-4.5);
\draw (0.75,-4)--(0.75,-4.5);
\draw (1.25,-4)--(1.25,-4.5);
\draw (1.75,-4)--(1.75,-4.5);
\draw (2.25,-4)--(2.25,-4.5);
\draw (2.75,-4)--(2.75,-4.5);
\draw (3.25,-4)--(3.25,-4.5);
\draw (3.75,-4)--(3.75,-4.5);
\node (1) at (0,0) {$\bullet$};
\node (1) at (-2,-1) {$\bullet$};
\node (1) at (2,-1) {$\bullet$};
\node (1) at (2,-1.4) {$\cdots$};
\node (1) at (-2,-1.4) {$\cdots$};
\node (1) at (-3,-3) {$\cdots$};
\node (1) at (-1,-3) {$\cdots$};
\node (1) at (1,-3) {$\cdots$};
\node (1) at (3,-3) {$\cdots$};
\node (1) at (-3.5,-3.8) {$\cdots$};
\node (1) at (-2.5,-3.8) {$\cdots$};
\node (1) at (-1.5,-3.8) {$\cdots$};
\node (1) at (-0.5,-3.8) {$\cdots$};
\node (1) at (0.5,-3.8) {$\cdots$};
\node (1) at (1.5,-3.8) {$\cdots$};
\node (1) at (2.5,-3.8) {$\cdots$};
\node (1) at (3.5,-3.8) {$\cdots$};
\node (1) at (0,-1) {$\cdots$};
\node (1) at (-3.75,-4.75) {$\vdots$};
\node (1) at (-3.25,-4.75) {$\vdots$};
\node (1) at (-2.75,-4.75) {$\vdots$};
\node (1) at (-2.25,-4.75) {$\vdots$};
\node (1) at (-1.75,-4.75) {$\vdots$};
\node (1) at (-1.25,-4.75) {$\vdots$};
\node (1) at (-0.75,-4.75) {$\vdots$};
\node (1) at (-0.25,-4.75) {$\vdots$};
\node (1) at (0.25,-4.75) {$\vdots$};
\node (1) at (0.75,-4.75) {$\vdots$};
\node (1) at (1.25,-4.75) {$\vdots$};
\node (1) at (1.75,-4.75) {$\vdots$};
\node (1) at (2.25,-4.75) {$\vdots$};
\node (1) at (2.75,-4.75) {$\vdots$};
\node (1) at (3.25,-4.75) {$\vdots$};
\node (1) at (3.75,-4.75) {$\vdots$};
\node (1) at (2,-6.1) {\color{blue}$\cdots$};
\node (1) at (2,-2) {\color{red}$\cdots$};

\node (1) at (-3,-2) {\color{red}$\bullet$};
\node (v1) at (-3.3,-2) {\color{red}$v_1$};
\draw[blue,dotted,thick] (-3,-3.8)ellipse(0.96cm and 2.1cm);
\node (A1) at (-3,-6.1){\color{blue}$A_1$};
\draw (-3,-2)--(-3.5,-3)--(-3.25,-4);
\draw (-3.5,-3)--(-3.75,-4);
\draw (-2.25,-4)--(-2.5,-3)--(-2.75,-4);
\draw (-3,-2)--(-2.5,-3);
\node (1) at (-3.5,-3) {$\bullet$};
\node (1) at (-2.5,-3) {$\bullet$};
\node (1) at (-3.75,-4){$\bullet$};
\node (1) at (-3.25,-4){$\bullet$};
\node (1) at (-2.75,-4){$\bullet$};
\node (1) at (-2.25,-4){$\bullet$};

\node (v2) at (-1.3,-2) {\color{red}$v_2$};
\node (1) at (-1,-2) {\color{red}$\bullet$};
\draw[blue,dotted,thick] (-1,-3.8)ellipse(0.96cm and 2.1cm);
\node (A2) at (-1,-6.1){\color{blue}$A_2$};
\draw (-1,-2)--(-1.5,-3)--(-1.25,-4);
\draw (-1.5,-3)--(-1.75,-4);
\draw (-0.25,-4)--(-0.5,-3)--(-0.75,-4);
\draw (-1,-2)--(-0.5,-3);
\node (1) at (-1.5,-3) {$\bullet$};
\node (1) at (-0.5,-3) {$\bullet$};
\node (1) at (-1.75,-4){$\bullet$};
\node (1) at (-1.25,-4){$\bullet$};
\node (1) at (-0.75,-4){$\bullet$};
\node (1) at (-0.25,-4){$\bullet$};

\node (v3) at (0.7,-2) {\color{red}$v_3$};
\node (1) at (1,-2) {\color{red}$\bullet$};
\draw[blue,dotted,thick] (1,-3.8)ellipse(0.96cm and 2.1cm);
\node (A3) at (1,-6.1){\color{blue}$A_3$};
\draw (1,-2)--(0.5,-3)--(0.75,-4);
\draw (0.5,-3)--(0.25,-4);
\draw (1.75,-4)--(1.5,-3)--(1.25,-4);
\draw (1,-2)--(1.5,-3);
\node (1) at (0.5,-3) {$\bullet$};
\node (1) at (1.5,-3) {$\bullet$};
\node (1) at (0.25,-4){$\bullet$};
\node (1) at (0.75,-4){$\bullet$};
\node (1) at (1.25,-4){$\bullet$};
\node (1) at (1.75,-4){$\bullet$};

\node (v4) at (2.7,-2) {\color{red}$v_{k-1}$};
\node (1) at (3,-2) {\color{red}$\bullet$};
\draw[blue,dotted,thick] (3,-3.8)ellipse(0.96cm and 2.1cm);
\node (A4) at (3,-6.1){\color{blue}$A_{k-1}$};
\draw (3,-2)--(2.5,-3)--(2.75,-4);
\draw (2.5,-3)--(2.25,-4);
\draw (3.75,-4)--(3.5,-3)--(3.25,-4);
\draw (3,-2)--(3.5,-3);
\node (1) at (2.5,-3) {$\bullet$};
\node (1) at (3.5,-3) {$\bullet$};
\node (1) at (2.25,-4){$\bullet$};
\node (1) at (2.75,-4){$\bullet$};
\node (1) at (3.25,-4){$\bullet$};
\node (1) at (3.75,-4){$\bullet$};

\draw[blue,dotted,thick] (0,-0.75)ellipse(4cm and 0.9cm);
\node (V) at (0,0.5){\color{blue}$V\setminus(A_1\cup\cdots\cup A_{k-1})$};
\end{tikzpicture}
\caption{Choosing $v_1,\cdots,v_{k-1}$.}
\label{fig:pairwisenonadj}
\end{figure}

\vspace{0.3cm}

\noindent\underline{Step 3. Showing $h(V\setminus(A_1\cup\cdots\cup A_{k-1}))=h_k(G)$.}
\newline\noindent Firstly, note that we have \[h(V\setminus(A_1\cup\cdots\cup A_{k-1}))\le h_k(G).\] Indeed, since \[A_k\subseteq V\setminus(A_1\cup\cdots\cup A_{k-1}),\] we have \[h(V\setminus(A_1\cup\cdots\cup A_{k-1}))=\min_{C\subseteq V\setminus(A_1\cup\cdots\cup A_{k-1})}\varphi(C)\le\varphi(A_k)=h_k(G).\]

Next, for the sake of contradiction, we suppose that $h(V\setminus(A_1\cup\cdots\cup A_{k-1}))<h_k(G)$. Then, by the definition of $h$, there is some $A_k'\subseteq V\setminus(A_1\cup\cdots\cup A_{k-1})$ such that \[\varphi(A_k')=h(V\setminus(A_1\cup\cdots\cup A_{k-1}))<h_k(G).\] Then, we can replace $A_k$ by $A_k'$ and consider the subpartition \[\mathcal A':=(A_1',\cdots,A_{k-1}',A_k'):=(A_1,\cdots,A_{k-1},A_k'),\] where $A_i':=A_i$ for $1\le i\le k-1$. We divide into two cases:

\vspace{0.3cm}

\noindent\textsl{Case 1.} $\varphi(A_{k-1})=\varphi(A_k)=h_k(G)$.
\newline Then, $f(\mathcal A')=f(\mathcal A)-1<f(\mathcal A)$ contradicts the assumption that $\mathcal A\in S_3\subseteq S_2$.

\vspace{0.3cm}

\noindent\textsl{Case 2.} $\varphi(A_{k-1})<\varphi(A_k)=h_k(G)$.
\newline Then, \[\max_{1\le i\le k}\varphi(A_i')=\max\{\varphi(A_{k-1}),\varphi(A_k')\}<h_k(G)\] contradicts the assumption that $\mathcal A\in S_3\subseteq S_1$.

\vspace{0.3cm}

Therefore, we must have $h(V\setminus(A_1\cup\cdots\cup A_{k-1}))=h_k(G)$. 

\vspace{0.3cm}

\noindent\underline{Step 4. Showing $h(A_i\setminus\{v_i\})\ge h_k(G)\text{ for any }1\le i\le k-1\text{ such that }A_i\setminus\{v_i\}\ne\varnothing$.}
\newline\noindent Suppose on the contrary that, for some $1\le j\le k-1$, $h(A_j\setminus\{v_j\})<h_k(G)$. Then, by the definition of $h(\cdot)$, there exists a subset $\overline A_j\subseteq A_j\setminus\{v_j\}$ such that $h_k(G)>h(A_j\setminus\{v_j\})=\varphi(\overline A_j)$ we can replace the subpartition $\mathcal A$ by \[\overline{\mathcal A}=(A_1,\cdots,\overline A_j,\cdots,A_k)\in S_2.\] Note that $\overline{\mathcal A}\in S_1$ because the $\varphi$-value of each of the $k$ subsets in $\mathcal A$ does not exceed $h_k(G)$. Also, $\overline{\mathcal A}\in S_2$ because $f(\overline{\mathcal A})$ either equals $f(\mathcal A)$ or $f(\mathcal A)-1$ (when $\varphi(A_j)<h_k(G)$ or $\varphi(A_j)=h_k(G)$, respectively). However, we have $\mu(\overline{\mathcal A})=\mu(\mathcal A)-\mu_{v_j}<\mu(\mathcal A)$. This contradicts the assumption that $\mathcal A\in S_3$.

\vspace{0.3cm}

\noindent\underline{Step 5. Showing $\underline{\ell_k}(G)\ge h_k(G)$ for any weighted tree $G$.}
\newline\noindent Indeed, we have
\begin{center}\begin{align}\underline{\ell_k}(G)&=\max_{A\subseteq V,~|A|=n-k+1}h(A)\notag
\\&\ge h(V\setminus\{v_1,\cdots,v_{k-1}\})\notag
\\&=\min\{h(V\setminus(A_1\cup\cdots\cup A_{k-1})),h(A_i\setminus\{v_i\})\colon 1\le i\le k-1,A_i\setminus\{v_i\}\ne\varnothing\}\label{eq:h-decompose}
\\&=h(V\setminus(A_1\cup\cdots\cup A_{k-1}))\label{eq:hgehk}
\\&=h_k(G),\notag\end{align}\end{center}
where the equality \eqref{eq:h-decompose} is based on {Lemma \ref{lem:cheeger-set-disjoint}} and the pairwise-nonadjacency of the $k$ subsets, and the equality \eqref{eq:hgehk} is due to the result proved in the last step.

\vspace{0.3cm}

\noindent\underline{Step 6. Sketch for the general case when $G$ is a weighted forest.}
\newline\noindent We can actually apply the same idea as above. Namely, pick a subpartition from $S_3$ and remove $v_1\cdots,v_{k-1}$ in order to achieve the pairwise nonadjacency. If there are some remaining quotas of the $k-1$ quotas, we can safely remove any vertices that are unrelated to the subgraph that achieves the minimum $\varphi$-value.

\vspace{0.3cm}

\noindent\underline{Step 7. Showing $h_k(\tilde G)\ge\ell_k(\tilde G)$ for any weighted graph $\tilde G=(\tilde V,\tilde E,\tilde\mu,\tilde w)$.}
\newline\noindent It suffices to show that for any subpartitions $(A_1,\cdots,A_k)$ and $(B_1,\cdots,B_{n-k+1})$ of $\tilde V$, we have \[\max_{1\le i\le k}\varphi(A_i)\ge\min_{B\in \mathcal{U}(B_1,\cdots,B_{n-k+1})}\varphi(B).\] Let $|\tilde V|=n$, apply Lemma \ref{lem:pigeonhole} to take the desired common union \[C:=A_{i_1}\cup\cdots\cup A_{i_t}=B_{j_1}\cup\cdots\cup B_{j_s}\]
where $\{i_1,\cdots,i_t\}$ is a nonempty subset of $\{1,\cdots,k\}$, and $\{j_1,\cdots,j_s\}$ is a nonempty subset of $\{1,\cdots,n-k+1\}$. Then, 
\begin{align*}
\max_{1\le i\le k}\varphi(A_i)&\ge \max\{\varphi(A_{i_1}),\cdots,\varphi(A_{i_t})\}\ge \varphi(A_{i_1}\cup\cdots\cup A_{i_t})=\varphi(C)
\\&
=\varphi(B_{j_1}\cup\cdots\cup B_{j_s})
\ge \min_{B\in \mathcal{U}(B_1,\cdots,B_{n-k+1})}\varphi(B).
\end{align*}

\vspace{0.3cm}

\noindent\underline{Step 8. Conclusion.}
\newline\noindent Since it is obvious that $\ell_k(\tilde G)\ge\underline{\ell_k}(\tilde G)$ for any weighted graph $\tilde G$, the proof is established.

\end{proof}

\subsubsection{Proofs of Theorem \ref{thm:cycle-number-inequ} and Theorem \ref{thm:CombinatorialCourant-Fischer}}
\begin{proof}[Proof of Theorem \ref{thm:cycle-number-inequ}]
The proof is essentially the same as that of Theorem \ref{thm:h_k=H_k}. We only need to notice the fact that, in a connected component with $\beta$ independent cycles, it takes us to remove at most $s+\beta-1$ vertices to separate the induced subgraphs of $s$ pairwise nonadjacent subset of vertices. Now apply Lemma \ref{lem:cheeger-set-disjoint}.
\end{proof}

\begin{proof}[Proof of Theorem \ref{thm:CombinatorialCourant-Fischer}]
The combinatorial Courant-Fischer-Weyl minimax principle is a direct consequence of Proposition \ref{prop:unioninv} and Theorem \ref{thm:h_k=H_k}.
\end{proof}

\subsection{Proofs of the Main Applications (Analytical)}
\subsubsection{Tools and Properties of Graph $p$-Laplacian
} \label{sec:tools}
As stated in Section \ref{sec:refined-Cheeger}, to make the higher order Cheeger inequality concise, in this section we assume $\mu_v=\sum_{u\in V:u\sim v}w_{uv}$ for any $v\in V$.
Under this assumption, the second-named author of this paper established in \cite{Zhang} the monotonicity property of graph $p$-Laplacian as follows: 
\begin{lemma}[monotonicity lemma \cite{Zhang}]\label{lem:p-monotonicity}
Given a weighted graph $G=(V,E,\mu,w)$, assume that $\mu_v=\sum_{u\in V:u\sim v}w_{uv}$ for any $v\in V$. 
For any $k$, the $k$-th min-max eigenvalue $\lambda_k(\Delta_p)$ is locally Lipschitz continuous with respect to $p$, and moreover, 
\begin{itemize}
    \item the function $p\mapsto p(2\lambda_k(\Delta_p))^{\frac1p}$ 
is increasing  on $
[1,+\infty)$,
\item 
the function 
$p\mapsto2^{-p}\lambda_k(\Delta_p)$ is decreasing on $[1,+\infty)$.
\end{itemize}    
\end{lemma}

For simplicity, we use $\varphi_p$ and $\II_k$ to denote the $L^p$-Rayleigh quotient and the set of origin-symmetric compact subsets with index $\ge k$, respectively \cite{Burger,Deidda25,Deidda26,HN25}. Their precise definitions are as follows:

Let $$\varphi_p(x)=\frac{\sum\limits_{u\sim v}w_{uv}|x_u-x_v|^p}{\sum\limits_v\mu_v|x_v|^p} $$
be the $L^p$-Rayleigh quotient on a graph $G=(V,E,\mu,w)$, where $p\ge1$. 

Let $$\II_k=\{S\subset \R^n:~S\text{ is origin-symmetric and compact with }\mathrm{ind}(S)\ge k\}.$$

\subsubsection{Proofs of Theorem \ref{thm:1-Lap=h} and Theorem \ref{thm:nonlinearCourant-Fischer}}
We first establish an intersection property of admissible indexes. 
\begin{lemma}[intersection lemma]\label{lem:intersect}
For any linear subspace $X\subseteq\mathbb{R}^n$ of dimension $n-k+1$, and any origin-symmetric subset $S$ with $\mathrm{ind}(S)\ge k$, we have $S\cap X\ne\varnothing$. 
\end{lemma}

\begin{proof}[Proof of Lemma \ref{lem:intersect}]

In fact, suppose the contrary, that $S\cap X=\varnothing$. Without loss of generality, we may assume $X=\mathbb{R}^{n-k+1}$. Consider the projection $P:\mathbb{R}^n\to\mathbb{R}^n/X\cong \mathbb{R}^{k-1}$, which is an odd continuous map. Since $S\cap X=\varnothing$, the image $P(S)\not\ni0$, and thus $P$ induces an odd continuous map from $S$ to $\mathbb{R}^{k-1}\setminus\{0\}$. 
It follows from the nondecreasing property under odd continuous map $P$ that  $\mathrm{ind}(P(S))\ge \mathrm{ind}(S)\ge k$. However, $P(S)\subseteq \mathbb{R}^{k-1}\setminus\{0\}$ and thus by the monotonicity of index, we have $\mathrm{ind}(P(S))\le \mathrm{ind}(\mathbb{R}^{k-1})=k-1$, which is a contradiction. 
\end{proof}

The following lemma asserts that we can always select a subset of the subpartition $\{A_1,\cdots,A_k\}$ whose union has a $\varphi$-value smaller than or equal to $\varphi_1(x)$ for prescribed $x\in \mathrm{span}(\mathbbm 1_{A_1},\cdots,\mathbbm 1_{A_{k}})\setminus\{0\}$.
 
\begin{lemma}[selection lemma]\label{lem:levelset}
For any $x\in \mathrm{span}(\mathbbm 1_{A_1},\cdots,\mathbbm 1_{A_{n-k+1}})\setminus\{0\}$, there exists $B\in \mathcal{U}(A_1,\cdots,A_{n-k+1})$ such that $\varphi_1(x)\ge\varphi(B)$. 
\end{lemma}
\begin{proof}[Proof of Lemma \ref{lem:levelset}]
For any \[x=\sum\limits_{i=1}^{n-k+1}t_i\mathbbm 1_{A_i}\in\R^n\setminus\{0\},\] there exist $ x^+, x^-\in \R^n_+$ such that $ x= x^+- x^-$, where \[x^+=\sum\limits_{i\in\{1,\cdots,n-k+1\}:~t_i\ge 0}t_i\mathbbm 1_{A_i}\text{ and }x^-=\sum\limits_{i\in\{1,\cdots,n-k+1\}:~t_i\le 0}t_i\mathbbm 1_{A_i}.\] It is easy to check that 
\begin{equation}\label{eq:F_1x+x-}
 \varphi_1( x)=\frac{\sum\limits_{u\sim v}w_{uv}|x_u^+-x_v^+|+\sum\limits_{u\sim v}w_{uv}|x_u^--x_v^-|}{\sum\limits_{v\in V}\mu_v|x_v^+|+\sum\limits_{v\in V}\mu_v|x_v^-|}\ge\min\{\varphi_1( x^+),\varphi_1( x^-)\}.   
\end{equation} 
Therefore, to prove this lemma, it suffices to deal with the case when $ x\in \R^n_+\setminus\{0\}$. Using layer cake representation, it can be verified that
\[\varphi_1( x)=\frac{\int_0^{\| x\|_\infty}w(\partial B^t)dt}{\int_0^{\| x\|_\infty}\mu(B^t)dt}\ge \frac{w(\partial B^{t'})}{\mu(B^{t'})}=\varphi(B_{t'})\]
for some $0<t'<\| x\|_\infty:=\max\limits_{i=1,\cdots,n-k+1}t_i$, 
where $B^t
:=\{v\in V:x_v>t\}$. Clearly, 
$$B^{t'}=\bigcup_{i\in\{1,\cdots,n-k+1\}:t_i>t'}A_i\in \mathcal{U}(A_1,\cdots,A_{n-k+1}).$$
Then, we can easily take $B=B^{t'}$ to conclude $\varphi_1( x)\ge\varphi(B)$.
\end{proof}
In the following, we establish the key lemma of this section, which shows that the $k$-th minimax and max-min eigenvalues of 1-Laplacian lie between $h_k(G)$ and $\ell_k(G)$. 
\begin{lemma}\label{lem:general-graph-delta}
For any weighted graph $G=(V,E,\mu,w)$, we have
\begin{equation}\label{eq:hk>lambdak>lk}
h_k(G)\ge \lambda_k^{\mathrm{ind}}(\Delta_1)\ge \ell_k(G)
\end{equation}
and \begin{equation}\label{eq:hk>lambdak2>lk}
h_k(G)\ge \underline{\lambda_k^{\mathrm{ind}}}(\Delta_1)\ge \ell_k(G) \end{equation}
\end{lemma}

\begin{proof}[Proof of Lemma \ref{lem:general-graph-delta}]
First, note that $\varphi_1(\mathbbm 1_A)=\varphi(A)$ for any nonempty subset $A\subseteq V$, where $\mathbbm 1_A$ is the indicator vector of $A$. For any subpartition $(A_1,\cdots,A_k)$, the indicator vectors $\mathbbm 1_{A_1},\cdots,\mathbbm 1_{A_k}$ span a linear subspace of dimension $k$. Taking $S_k:=\mathrm{span}(\mathbbm 1_{A_1},\cdots,\mathbbm 1_{A_k})\cap\mathbb{S}^{n-1}$, it is clear that $S_k$ is the standard unit sphere of dimension $k-1$ and thus $\mathrm{ind}(S_k)=k$, i.e., $S_k\in\mathcal{I}_k$. 
Therefore, 
\begin{align*}
h_k(G)&=\min_{(A_1,\cdots,A_k)\in\mathcal P_k(V)}~\max_{1\le i\le k}\varphi_1(\mathbbm 1_{A_i})
\\&=\min_{(A_1,\cdots,A_k)\in\mathcal P_k(V)}~\sup_{(t_1,\cdots,t_k)\ne0}\varphi_1(t_1\mathbbm 1_{A_1}+\cdots+t_k\mathbbm 1_{A_k})
\\&\ge \inf_{S\in\mathcal{I}_k}~\sup_{x\in S}\varphi_1(x) =\lambda_k^{\mathrm{ind}}(\Delta_1).  
\end{align*}

For any subpartition $(A_1,\cdots,A_{n-k+1})$ reaching $\ell_k(G)$, we have \[\dim \mathrm{span}(\mathbbm 1_{A_1},\cdots,\mathbbm 1_{A_{n-k+1}})=n-k+1.\] Hence, for any $S\in\mathcal{I}_k$, it follows from Lemma \ref{lem:intersect} that $$S\cap \mathrm{span}(\mathbbm 1_{A_1},\cdots,\mathbbm 1_{A_{n-k+1}})\ne\varnothing.$$ Consequently, 
\begin{align}
\lambda_k^{\mathrm{ind}}(\Delta_1)&=\inf_{S\in\mathcal{I}_k}~\sup_{x\in S}\varphi_1(x) \notag
\\&\ge     \inf_{S\in\mathcal{I}_k}~\sup_{x\in S\cap \mathrm{span}(\mathbbm 1_{A_1},\cdots,\mathbbm 1_{A_{n-k+1}})}\varphi_1(x)  \notag
\\&\ge \inf_{x\in \mathrm{span}(\mathbbm 1_{A_1},\cdots,\mathbbm 1_{A_{n-k+1}})}\varphi_1(x) \notag
\\&\ge\min_{B\in \mathcal{U}(A_1,\cdots,A_{n-k+1})}\varphi(B) 
\label{eq:min=Che}
\\&=\ell_k(G) \notag
\end{align}
where \eqref{eq:min=Che} is based on Lemma \ref{lem:levelset}. 

We have then established the inequality 
\eqref{eq:hk>lambdak>lk}
for the weighted graph $G$. 
A similar process gives 
$$h_k\ge \sup_{\substack{S\text{ origin-symmetric, compact}\\\mathrm{ind}(S)\ge n-k+1}}~\min_{x\in S}\varphi_1(x)\ge \ell_k(G),$$
which establishes the inequality 
\eqref{eq:hk>lambdak2>lk}.
\end{proof}

\begin{proof}[Proof of Theorem \ref{thm:1-Lap=h}]
The equality $h_k(G)=\lambda_k^{\mathrm{ind}}(\Delta_1)=\underline{\lambda_k^{\mathrm{ind}}}(\Delta_p)=\ell_k(G)$ is a direct consequence of the equality $h_k(G)=\ell_k(G)$ for forests (as shown in Theorem \ref{thm:h_k=H_k}) and the inequalities \eqref{eq:hk>lambdak>lk} and \eqref{eq:hk>lambdak2>lk} (see Lemma \ref{lem:general-graph-delta}). 
\end{proof}

\begin{proof}[Proof of Theorem \ref{thm:nonlinearCourant-Fischer}]
Since the function $\mathbb{R}^n\setminus\{0\}\ni x\mapsto \varphi_1(x)$ is zero-homogeneous, any linear subspace $X$ can be replaced by its unit sphere $S:=X\cap \mathbb{S}^{n-1}$ such that $\dim X=\mathrm{ind}(X\cap \mathbb{S}^{n-1})=\mathrm{ind}(S)$ with nothing changes, i.e., 
$$\min_{\substack{\text{linear subspaces }X\text{ of }\mathbb R^n\\\dim X=k}}~\max_{x\in X\setminus\{0\}} \varphi_1(x)=\min_{\substack{S=X\cap \mathbb{S}^{n-1}\\\text{linear subspaces }X\text{ of }\mathbb R^n\\\dim X= k}}~\max_{x\in S}\varphi_1(x) .$$
Replacing $\inf\limits_{S\in\mathcal{I}_k}~\sup\limits_{x\in S}\varphi_1(x)$ in the proof of Lemma \ref{lem:general-graph-delta} by $\inf\limits_{\dim X=k}~\sup\limits_{x\in X}\varphi_1(x)$ yields the similar inequality
\begin{align*}h_k(G)&\ge\min\limits_{\substack{\text{linear subspaces }X\text{ of }\mathbb R^n\\\dim X= k}}~\max\limits_{x\in X\setminus\{0\}} \frac{\sum\limits_{u\sim v}w_{uv}|x_u-x_v|}{\sum\limits_v\mu_v|x_v|}\\&\ge \max\limits_{\substack{\text{linear subspaces }X\text{ of }\mathbb R^n\\\dim X= n-k+1}}~\min\limits_{x\in X\setminus\{0\}} \frac{\sum\limits_{u\sim v}w_{uv}|x_u-x_v|}{\sum\limits_v\mu_v|x_v|}
\ge \ell_k(G) .\end{align*}
Since $G$ is a forest, we can use the equality $h_k(G)=\ell_k(G)$ in Theorem \ref{thm:h_k=H_k} to derive the nonlinear Courant-Fischer-Weyl minimax equality. 
\end{proof}

\subsubsection{Proofs of Other Results in Sections \ref{sec:minmax-eigen-1-Lap} and \ref{sec:refined-Cheeger}
}

\begin{proof}[Proof of Corollary \ref{cor:minmax-indepen-index}]
Since every index $\mathrm{ind}\in\{\gamma,\gamma^+,Y\textrm{-ind}\}$ is admissible, by Theorem \ref{thm:1-Lap=h} we obtain $\lambda_k^{\gamma}(\Delta_1)=\lambda_k^{\gamma^+}(\Delta_1) =\lambda_k^{\textrm{Y-ind}}(\Delta_1)$. 
\end{proof}

\begin{proof}[Proof of Corollary \ref{cor:tree-1-Lap-h}]
Taking $\mathrm{ind}=\gamma$ in Theorem \ref{thm:1-Lap=h}, we immediately derive $h_k(G)=\lambda_k^{\gamma}(\Delta_1)$. 
\end{proof}

\begin{proof}[Proof of Corollary \ref{cor:tree-p-Lap}]
We can express the two monotonicity inequalities in  Lemma \ref{lem:p-monotonicity} as
$$
\swarrow\frac{1}{2^{p-1}}\lambda_k(\Delta_p)\le \lambda_k(\Delta_p)\le \frac p2 (2\lambda_k(\Delta_p))^{\frac1p}\nearrow
$$
where the left-hand term $\frac p2 (2\lambda_k(\Delta_p))^{\frac1p}$ is increasing with respect to $p\in[1,+\infty)$, while the right-hand term $\frac{1}{2^{p-1}}\lambda_k(\Delta_p)$ is decreasing with respect to $p\in[1,+\infty)$. 
Substituting the equality $h_k(G)=\lambda_k(\Delta_1)$ into the above inequality immediately yields the conclusion of Corollary \ref{cor:tree-p-Lap}. 
\end{proof}

\begin{proof}[Proof of Theorem  \ref{thm:beta}]
By Lemma \ref{lem:general-graph-delta}, $h_k(G)\ge \lambda_k(\Delta_1)\ge \ell_k(G)$ for any graph $G$ and any $k$. By Theorem \ref{thm:h_k=H_k}, we have $\ell_k(G)\ge \underline{\ell_k}(G)$, and according to Theorem \ref{thm:cycle-number-inequ}, we have $\underline{\ell_k}(G)\ge h_{k-\beta}(G)$. Thus, we finally obtain $\lambda_k(\Delta_1)\ge h_{k-\beta}(G)$, which completes the proof of Theorem \ref{thm:beta}.
\end{proof}

\begin{proof}[Proof of Theorem \ref{thm:beta-p}] To give a proof of Theorem \ref{thm:beta-p}, we require the monotonicity of $p$-Laplacian eigenvalues, see Lemma \ref{lem:p-monotonicity} in Section \ref{sec:tools}. In fact, we shall use the inequalities \[p(2\lambda_k(\Delta_p))^{\frac1p}\ge 2\lambda_k(\Delta_1)\text{ and }2^{-p}\lambda_k(\Delta_p)\le 2^{-1}\lambda_k(\Delta_1)\] in Lemma \ref{lem:p-monotonicity}. Combining these two inequalities with Theorem \ref{thm:beta}, we derive
$$2^{1-p}\lambda_k(\Delta_p)\le \lambda_k(\Delta_1)\le h_k(G)$$
and 
$$p(2\lambda_k(\Delta_p))^{\frac1p}\ge 2\lambda_k(\Delta_1)\ge 2h_{k-\beta}(G).$$
This proves Theorem \ref{thm:beta-p}.
\end{proof}

\begin{proof}[Proof of Corollary \ref{cor:unicyclic}]
By definition, a graph is unicyclic if and only if $\beta=1$. Then, we conclude the proof by taking $\beta=1$ in Theorem \ref{thm:beta-p}.
\end{proof}

\vspace{0.6cm}

\noindent{\bf Acknowledgements:} 
D. Zhang is supported by NSFC (no. 12401443).

\vspace{0.6cm}

\noindent\textbf{AI statement}. We used GPT-5.5 to polish the introduction section of this paper. We assume full responsibility for the content and correctness of the paper.

\end{document}